\documentclass[11pt]{article}

\usepackage{amsmath}
\usepackage{amsthm}
\usepackage{amssymb}
\usepackage{latexsym}
\usepackage{textcomp}
\usepackage[T1]{fontenc}
\usepackage[sc]{mathpazo}
\usepackage{enumerate}
\usepackage{titlefoot}
\usepackage{titlesec}
\usepackage[small,it]{caption}
\usepackage{pstricks}
\usepackage{graphicx, pst-plot, pst-node, pst-text, pst-tree}
\usepackage{color}
\definecolor{lightgray}{rgb}{0.9, 0.9, 0.9}
\definecolor{darkgray}{rgb}{0.7, 0.7, 0.7}
\definecolor{darkblue}{rgb}{0, 0, .4}

\usepackage[bookmarks]{hyperref}
\hypersetup{
    colorlinks=true,
    linkcolor=darkblue,
    anchorcolor=darkblue,
    citecolor=darkblue,
    urlcolor=darkblue,
    pdfpagemode=UseThumbs,
    pdftitle={Enumerating Av(4231,2143)},
    pdfsubject={Combinatorics},
    pdfauthor={M. Albert, M. Atkinson and R. Brignall},
    pdfkeywords={todo}
}
\theoremstyle{plain}
\newtheorem{theorem}{Theorem}[section]
\newtheorem{corollary}[theorem]{Corollary}
\newtheorem{lemma}[theorem]{Lemma}
\newtheorem{proposition}[theorem]{Proposition}

\theoremstyle{definition}

\theoremstyle{remark}

\newcounter{todocounter}

\makeatletter
\newfont{\footsc}{cmcsc10 at 8truept}
\newfont{\footbf}{cmbx10 at 8truept}
\newfont{\footrm}{cmr10 at 10truept}
\newcommand{\Av}{\operatorname{Av}}

\newcommand{\D}{\mathcal{D}}
\newcommand{\E}{\mathcal{E}}
\newcommand{\F}{\mathcal{F}}

\renewcommand{\P}{\mathcal{P}}

\date{}

\title{The enumeration of permutations avoiding 2143 and 4231}
\author{
M.~H. Albert\thanks{{\tt malbert@cs.otago.ac.nz}},
M.~D. Atkinson\thanks{{\tt mike@cs.otago.ac.nz}}\\
Department of Computer Science\\
University of Otago,  New Zealand\\
and\\
Robert Brignall\thanks{{\tt r.brignall@open.ac.uk}}\\
Department of Mathematics\\
The Open University, UK}

\begin{document}
\maketitle

\begin{abstract}We enumerate the pattern class $\Av(2143,4231)$ and completely describe its permutations.  The main tools are simple permutations and monotone grid classes.
\end{abstract}

\section{Introduction}

This paper is a contribution to an ongoing endeavour initiated by Knuth in section 2.2.1 of \cite{knuth:the-art-of-comp:1}: in Exercises 7--11 (and their solutions)   he enumerated the permutations that can be obtained from an input-restricted deque by first proving that they are exactly those that do not contain either of the patterns 4213 and 4231.  Since that time the subject of Permutation Patterns has developed into a rich combinatorial theory one of whose central concerns continues to be the description and enumeration of permutations that do not contain a stipulated set of patterns.

For completeness we briefly recall the salient definitions.  A permutation is simply an arrangement of the numbers $1,2,\ldots n$ for some $n>0$ (note therefore that all our permutations will be non-empty).  A permutation $\pi$ is said to be contained in (or be a subpermutation of) another permutation $\sigma$ if $\sigma$ has a subsequence whose terms are ordered in the same relative way as those of $\pi$.  For example, 3142 is contained in  1573462 because the subsequence 5362 is ordered in the same way as 3142.  If $\pi$ is not contained in $\sigma$ we say that $\sigma$ \emph{avoids} $\pi$.  The subpermutation relation is obviously a partial order on the set of all permutations and its down-sets are called {\em pattern classes}.  For every pattern class $\P$ there is a (possibly infinite) set of permutations that do not belong to $\P$ and are minimal with respect to not lying in $\P$.  This set, $B$ say, is called the \emph{basis} of $\P$ and it determines $\P$ as exactly that set of permutations that avoid every member of $B$: we write $\P=\Av(B)$.

The subpermutation order is invariant under the 8 symmetries generated by inversion, reversal and complementation.  These symmetries often allow arguments by case enumeration to be condensed.

If the set $B$ contains any permutation of length 1 or 2 then it is trivial to identify $\Av(B)$ and to enumerate it.  If $B$ contains 132 (or any of its 3 symmetries 213, 231, or 312) then $\Av(B)$ can be enumerated by the methods of \cite{albert:simple-permutat:}.  If $B$ is only known to contain 123 (or its symmetry 321) much less has been proved although many special cases have been solved (see, for example, \cite{albert:growth-rates-for-subclasses-of-Av(321):, atkinson:on-permutation-:, mansour:321-polygon-avo:, stankova:explicit-enumer:}).

Knuth's problem above  was the first one to be solved with $B$ having two permutations both of length 4.  There are in fact 56 essentially different such problems (i.e. problems in which the sets $B$ are inequivalent under symmetries).  It is known that they give rise to $38$ different enumerations \cite{bona:the-permutation:,kremer:permutations-wi:, kremer:postscript:-per:, kremer:finite-transiti:, le:wilf-classes-of:} (some inequivalent pairs are {\em Wilf-equivalent} meaning that they nevertheless have the same enumeration).  Of these $38$ Wilf classes around half have yet to be enumerated (see \cite{wiki:enumerations} which lists 18 published enumerations).

The ones that have been enumerated have been natural testing grounds for a succession of techniques (such as generating trees \cite{west:generating-tree1:}, the insertion encoding \cite{albert:the-insertion-e:}, the Schensted correspondence \cite{atkinson:permutations-wh:}, and simple permutations \cite{albert:simple-permutat:}).  In this paper we apply a combination of recently devised techniques to the pattern class $\Av(2143, 4231)$ and, not only do we enumerate it, we give a complete structural description of its members.

Our main technical tools are simple permutations and their inflations, together with  constrained decompositions of permutations into monotone subsequences and we now summarise the basic facts we shall need about these.  

A \emph{simple} permutation is one with no non-trivial intervals.  In this context an \emph{interval} of a permutation  is just a contiguous subsequence whose values form a consecutive set of integers.  For example 546 is an interval of 3154627.  If the interval is either a singleton or the entire permutation then it is trivial.  
Simple permutations are precisely those that do not arise from a non-trivial inflation, in the following sense.  Let $\sigma$ be any permutation of length $m$ and $\alpha_1,\alpha_2,\ldots,\alpha_m$ any sequence of permutations.  Then the \emph{inflation} of $\sigma$ by $\alpha_1,\alpha_2,\ldots,\alpha_m$, which we denote by  $\sigma[\alpha_1,\alpha_2,\ldots,\alpha_m]$, is that permutation of length $|\alpha_1|+\cdots+|\alpha_m|$ which decomposes into $m$ segments $\alpha'_1\alpha'_2\cdots\alpha'_n$ where each segment $\alpha'_i$ is an interval that is order isomorphic to $\alpha_i$, and the sequence $a_1a_2\cdots a_n$ formed by any (and hence every) choice of $a_i$ from $\alpha'_i$ is order isomorphic to $\sigma$.  For example the inflation of $3142$ by $21, 132, 1, 123$ is
\[3142[21, 132, 1, 123]=87\ 132\ 9\ 456\]

Permutations that are inflations of 12 and 21, which occur often in this paper, are said to be, respectively, {\em sum decomposable} and {\em skew decomposable}.  
The precise connection between simple permutations and inflations is furnished by a result from \cite{albert:simple-permutat:}.

\begin{proposition}
\label{CanonicalDecomposition}
Let $\pi$ be any permutation. Then there is a unique simple permutation $\sigma$ and permutations $\alpha_1,\ldots,\alpha_n$ such that
\[
\pi = \sigma [\alpha_1,\ldots,\alpha_n].
\]
If $\sigma \neq 12, 21$, then $\alpha_1,\ldots,\alpha_n$ are also uniquely
determined by $\pi$.  If $\pi = 12$ or $21$, then $\alpha_1,
\alpha_2$
are unique so long as we require that $\alpha_1$ is
sum indecomposable
or skew indecomposable respectively.
\end{proposition}

Our other technical tool is a diagrammatic expositional aid.  We regard a permutation $\pi$ as a set of points $(i,\pi(i))$ lying in an $n\times n$ grid within the plane.  We partition such a square grid into cells using a fixed number of vertical and horizontal dividing lines.  The points that lie within a cell then define a subsequence of $\pi$ and we shall be particularly interested in when these subsequences are monotone.  Consider, for example, the permutation $[6,12,11,7,10,4,5,9,3,8,2,1]$.  As shown in Figure \ref{fig-grid-example} we can represent it on a $3\times 2$ grid with one empty cell, two increasing cells, and three decreasing cells.

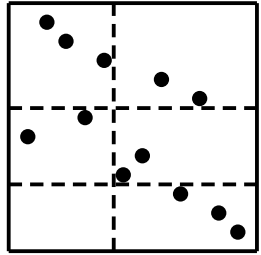
\begin{figure}
\begin{center}
\psset{xunit=0.01in, yunit=0.01in} \psset{linewidth=0.005in}
\begin{pspicture}(0,0)(140,140)
\psline[linestyle=solid,linewidth=0.02in](0,0)(0,130)
\psline[linestyle=solid,linewidth=0.02in](0,0)(130,0)
\psline[linestyle=solid,linewidth=0.02in](130,0)(130,130)
\psline[linestyle=solid,linewidth=0.02in](0,130)(130,130)
\psline[linestyle=dashed,linewidth=0.02in](0,35)(130,35)
\psline[linestyle=dashed,linewidth=0.02in](0,75)(130,75)
\psline[linestyle=dashed,linewidth=0.02in](55,0)(55,130)
\pscircle*(10,60){0.04in}
\pscircle*(20,120){0.04in}
\pscircle*(30,110){0.04in}
\pscircle*(40,70){0.04in}
\pscircle*(50,100){0.04in}
\pscircle*(60,40){0.04in}
\pscircle*(70,50){0.04in}
\pscircle*(80,90){0.04in}
\pscircle*(90,30){0.04in}
\pscircle*(100,80){0.04in}
\pscircle*(110,20){0.04in}
\pscircle*(120,10){0.04in}
\end{pspicture}
\end{center}
\caption{The permutation $[6,12,11,7,10,4,5,9,3,8,2,1]$ on a $3\times 2$ grid}
\label{fig-grid-example}
\end{figure}

The set of permutations whose diagram can be divided into cells with a fixed number of vertical and horizontal lines, where it is stipulated that each cell should contain either monotone increasing points, monotone decreasing points, a single point, or be empty, defines a pattern class, called a \emph{grid} class.  We may denote such a grid class by specifying its matrix of cells with each cell containing $+1$ for increasing, $-1$ for decreasing, a dot for a single point and $0$ for empty.  It is convenient sometimes to denote the content of an increasing (respectively decreasing) cell by a line ascending (respectively descending) to the right.  So the permutation $[6,12,11,7,10,4,5,9,3,8,2,1]$ belongs to the grid class whose matrix is
\[
\left( \begin{array}{cc}
-1&-1\\
+1&+1\\
0&-1
\end{array}
\right)
\]
or whose pictorial representation is
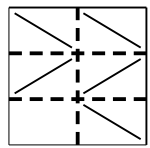
\begin{center}
\psset{xunit=0.006in, yunit=0.006in} \psset{linewidth=0.005in}
\begin{pspicture}(0,0)(130,130)
\psaxes[dy=200,dx=200,tickstyle=bottom,showorigin=false,labels=none](0,0)(120,120)
\psline[linestyle=solid,linewidth=0.01in](0,120)(120,120)
\psline[linestyle=solid,linewidth=0.01in](120,0)(120,120)
\psline[linestyle=dashed,linewidth=0.02in](0,40)(120,40)
\psline[linestyle=dashed,linewidth=0.02in](0,80)(120,80)
\psline[linestyle=dashed,linewidth=0.02in](60,0)(60,120)
\psline[linestyle=solid,linewidth=0.01in](65,35)(115,5)
\psline[linestyle=solid,linewidth=0.01in](5,45)(55,75)
\psline[linestyle=solid,linewidth=0.01in](65,45)(115,75)
\psline[linestyle=solid,linewidth=0.01in](5,115)(55,85)
\psline[linestyle=solid,linewidth=0.01in](65,115)(115,85)
\end{pspicture}
\end{center}

The bulk of our paper is an analysis of $\Av(2143, 4231)$.  We determine its simple permutations and describe them in grid class terminology.  Next we examine how the simple permutations  can be inflated and thereby we obtain a complete description of the pattern class.  The enumeration calculation is then carried out using encodings of permutations derived from the grid class description.  The paper ends with some remarks on pattern classes for which similar analyses may be possible.
\section{The structure of  \boldmath{$\Av(2143,4231)$}}

\subsection{The simple permutations}

We shall obtain the general form of simple permutations in the class by a division into cases according to the pattern determined by the four extremal points of the permutation.  Let these points be denoted by $\ell$ the leftmost point, $r$ the rightmost, $u$ the highest and $d$ the lowest.  In a simple permutation of length at least 4 these points are all distinct and their pattern will therefore be one of 2143 (which is impossible as it is a basis element of the class), 3412, 2413 and 3142.

\begin{lemma}
If $\pi$ is a simple permutation in  $\Av(2143, 4231)$ and the pattern determined by $\ell$, $r$, $u$ and $d$ is 3412 then $\pi$ is one of 42513 or 35142.
\end{lemma}

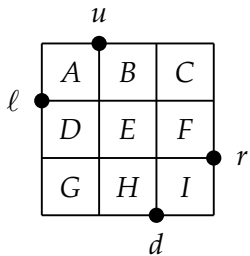
\begin{figure}
\begin{center}
\psset{xunit=0.03in, yunit=0.03in} \psset{linewidth=0.005in}
\begin{pspicture}(0,0)(60,60)
\multido{\i=10+10}{4}{%
\psline[linestyle=solid,linewidth=0.01in](10,\i)(40,\i)
\psline[linestyle=solid,linewidth=0.01in](\i,10)(\i,40)}
\pscircle*(10,30){0.04in}
\pscircle*(20,40){0.04in}
\pscircle*(30,10){0.04in}
\pscircle*(40,20){0.04in}
\rput[c](15,35){$A$}
\rput[c](25,35){$B$}
\rput[c](35,35){$C$}
\rput[c](15,25){$D$}
\rput[c](25,25){$E$}
\rput[c](35,25){$F$}
\rput[c](15,15){$G$}
\rput[c](25,15){$H$}
\rput[c](35,15){$I$}
\rput[c](5,30){$\ell$}
\rput[c](20,45){$u$}
\rput[c](30,5){$d$}
\rput[c](45,20){$r$}

\end{pspicture}
\end{center}
\caption{The extremal points form the pattern 3412}
\label{First3412}
\end{figure}


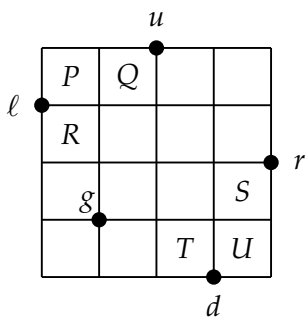
\begin{figure}
\begin{center}
\psset{xunit=0.03in, yunit=0.03in} \psset{linewidth=0.005in}
\begin{pspicture}(0,0)(60,60)
\multido{\i=10+10}{5}{%
\psline[linestyle=solid,linewidth=0.01in](10,\i)(50,\i)
\psline[linestyle=solid,linewidth=0.01in](\i,10)(\i,50)}
\pscircle*(10,40){0.04in}
\pscircle*(20,20){0.04in}
\pscircle*(30,50){0.04in}
\pscircle*(40,10){0.04in}
\pscircle*(50,30){0.04in}
\rput[c](35,15){$T$}
\rput[c](45,15){$U$}
\rput[c](45,25){$S$}
\rput[c](15,35){$R$}
\rput[c](15,45){$P$}
\rput[c](25,45){$Q$}
\rput[c](5,40){$\ell$}
\rput[c](30,55){$u$}
\rput[c](40,5){$d$}
\rput[c](55,30){$r$}
\rput[c](18,23){$g$}
\end{pspicture}
\end{center}
\caption{The extremal points and an interior point form the pattern $42513$.}
\label{fig-42513}
\end{figure}

\begin{proof}
Consider the regions of $\pi$ depicted in Figure \ref{First3412}.  By the 2143 avoidance one of the regions $B\cup C$ and $D\cup G$ is empty.  

To within a reverse complement symmetry we may take $B\cup C$ to be empty.  But then region $G$ is not empty.  If it were then, because $\ell A u$ is not an interval, $D$ is not empty; however all points of region $E\cup F$ lie below all points of region $D$ as 4231 is not a subpermutation of $\pi$; and now it would follow that the points of $A\cup D$ together with $\ell$ and $u$ would be an interval.

Thus $G$ contains some point $g$ and the four extremal points of $\sigma$ together with $g$ give rise to a 42513 pattern. We are now in the situation depicted in Figure~\ref{fig-42513}. Of the 16 square regions, the 10 unlabeled regions must be empty in order to avoid either 4231 or 2143. We claim that in order for $\sigma$ to be simple, the remaining six labeled regions must be empty. To justify this claim we begin by noting that every point of $P$ must lie below every point in $Q$, as otherwise we would find a 4231-pattern, using the points $d$ and $g$. Thus the points in region $Q$, together with $u$, form an interval, and since $\sigma$ is simple we infer that $Q$ must be empty. Now, however, the two regions $P$ and $R$, together with the point $\ell$ form an interval, from which we conclude that both $P$ and $R$ must be empty. The regions $S,T$ and $U$ must also be empty by a similar argument (or by taking inverses). Hence $\sigma=42513$.  Similarly, when $D\cup G$ is empty, we have $\sigma=35142$.
\end{proof}

\begin{lemma}\label{extremal3142}
If $\pi$ is a simple permutation in  $\Av(2143, 4231)$ and the pattern determined by $\ell$, $r$, $u$ and $d$ is 3142 then $\pi$ has the form shown in Figure \ref{fig-3142-case-result}.  In this figure the letters $A, B, C, D, E$ label the cells that contain them.  Unlabelled cells are empty.  Furthermore cells $C$ and $E$ are empty and possibly cells $A$ and $G$ are empty also.
\end{lemma}

\begin{figure}
\begin{center}
\psset{xunit=0.06in, yunit=0.06in} \psset{linewidth=0.005in}
\begin{pspicture}(0,5)(45,45)
\multido{\i=10+10}{4}{%
\psline[linestyle=solid,linewidth=0.01in](10,\i)(40,\i)
\psline[linestyle=solid,linewidth=0.01in](\i,10)(\i,40)}
\pscircle*(10,30){0.04in}
\pscircle*(20,10){0.04in}
\pscircle*(30,40){0.04in}
\pscircle*(40,20){0.04in}
\pscircle*(15,34){0.04in}
\pscircle*(35,16){0.04in}
\pscircle*(22,32){0.04in}
\pscircle*(28,18){0.04in}
\psline[linestyle=solid,linewidth=0.02in](24,28)(26,22)
\psline[linestyle=solid,linewidth=0.02in](25,35)(30,40)
\psline[linestyle=solid,linewidth=0.02in](20,10)(25,15)

\rput[c](8,30){$\ell$}
\rput[c](20,8){$d$}
\rput[c](30,42){$u$}
\rput[c](42,20){$r$}
\rput[c](12,38){$A$}
\rput[c](22,38){$B$}
\rput[c](12,28){$C$}
\rput[c](22,28){$D$}
\rput[c](32,28){$E$}
\rput[c](22,18){$F$}
\rput[c](32,18){$G$}

\end{pspicture}
\end{center}
\caption{Simple permutations where the extremal points form the pattern 3142}
\label{fig-3142-case-result}
\end{figure}
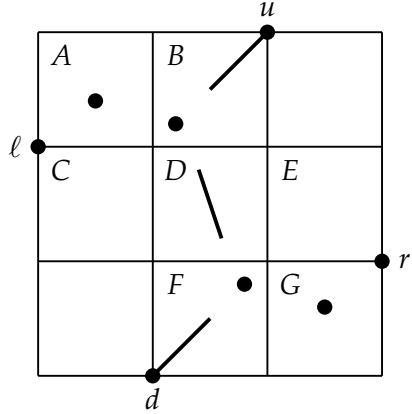

\begin{proof}

%
%

The initial situation is depicted in Figure~\ref{fig-3142-case} where the two unlabeled regions are empty since $\pi$ avoids $2143$.  The rectangle comprising cells $C,D,E$ must form a decreasing pattern to avoid $4231$: in particular, this means that points in the cell labeled $C$ lie above all points in cells $D$ and $E$. Thus, to avoid a non-trivial interval being formed by the point $\ell$ together with the points in cells $A$ and $C$, there must exist some point $a\in A$ lying above some point $b\in B$. However, any point in $C$ to the left of $a$ now participates as the `1' in a 2143 with $\ell$, $a$ and $b$, while any point in $C$ to the right of $a$ forms the `2' of a 4231 with $a$, $b$ and $r$. Thus $C$ must be empty.

\begin{figure}
\begin{center}
\psset{xunit=0.03in, yunit=0.03in} \psset{linewidth=0.005in}
\begin{pspicture}(0,0)(50,50)
\multido{\i=10+10}{4}{%
\psline[linestyle=solid,linewidth=0.01in](10,\i)(40,\i)
\psline[linestyle=solid,linewidth=0.01in](\i,10)(\i,40)}
\pscircle*(10,30){0.04in}
\pscircle*(20,10){0.04in}
\pscircle*(30,40){0.04in}
\pscircle*(40,20){0.04in}
\rput[c](15,35){$A$}
\rput[c](25,35){$B$}
\rput[c](15,25){$C$}
\rput[c](25,25){$D$}
\rput[c](35,25){$E$}
\rput[c](25,15){$F$}
\rput[c](35,15){$G$}
\rput[c](5,30){$\ell$}
\rput[c](20,5){$d$}
\rput[c](30,45){$u$}
\rput[c](45,20){$r$}
\end{pspicture}
\end{center}
\caption{The extremal points form the pattern 3142}
\label{fig-3142-case}
\end{figure}
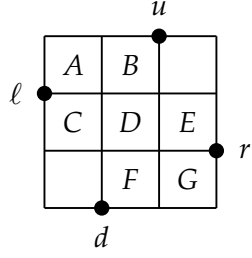

Because of the 2143-avoidance $B$ is increasing.  It cannot be empty except in a trivial case where $A$, which would now form an interval with $\ell$, is also empty; so we can let $b$ denote the lowest point of $B$.  The points of $A$ must all lie above $b$; for, if $A$ contained points below $b$, then these points together with $\ell$ would have to be separated by a point of $A$ that was greater than $b$ and then there would be a 4231 pattern.  Furthermore no point of $A$ can lie above the second largest point of $B$ (if $B$ has such a point) because, again, a 4231 pattern would be created.  But now the points of $A$ form an interval and so $A$ has at most one point ( which, when it exists, lies in value between the lowest and second lowest points of $B$).


A symmetric argument applies to the regions $E$ (which must be empty), $F$ (which contains an increasing sequence) and $G$ (which contains at most one point, below exactly one point of $F$), and from this we conclude that the permutation has the form given in Figure~\ref{fig-3142-case-result}.
\end{proof}

In the remaining case that the pattern of the extremal points is 2413 (which is the inverse of 3142) we have

\begin{corollary}\label{extremal2413}
If $\pi$ is a simple permutation in  $\Av(2143, 4231)$ and the pattern determined by $\ell$, $r$, $u$ and $d$ is 2413 then  $\pi^{-1}$ has the form shown in Figure \ref{fig-3142-case-result}.
\end{corollary}
%
%

\subsection{Inflating the simple permutations}

Now we determine which inflations of simple permutations lie in $\Av(2143, 4231)$ and we start with the simple permutations of Lemma \ref{extremal3142}.
Let $\D$ be the grid class whose matrix is the column vector$\left(\begin{array}{c}1\\-1\\1\end{array}\right)$.  Then, by Lemma \ref{extremal3142}, a simple permutation $\sigma$ whose extremal points have the pattern 3142 has one of four types:
\begin{itemize}
\item $\sigma \in \D$ --- we will call these permutations \emph{type 1} simple permutations,
\item $\sigma \not\in \D$, but $\sigma - \sigma(1)\in \D$ --- \emph{type 2},
\item $\sigma \not\in \D$, but $\sigma - \sigma(n)\in \D$ --- \emph{type 3}, or
\item $\sigma \not\in \D$ and not of types 2 or 3, but $\sigma-\sigma(1)-\sigma(n)\in \D$ --- \emph{type 4}.
\end{itemize}

These types correspond, respectively, to cells $A$ and $G$ (as defined in Lemma \ref{extremal3142}) having sizes $(0,0), (1,0), (0,1)$ and $(1,1)$.  The cell $B$ together with $u$, the cell $D$, and the cell $F$ together with $d$ (also defined in Lemma \ref{extremal3142}) are the three cells of $\D$, read from top to bottom.
Note that types 2 and 3 can be obtained from each other by the reverse complement symmetry, so we will handle these two cases together.
We will let $\E = \Av(2143,312)$ and $\mathcal{F} = \Av(2143,231)$; these classes have  grid class descriptions as depicted in Figure \ref{E&F}.

\begin{figure}
\begin{center}
\psset{xunit=0.03in, yunit=0.03in} \psset{linewidth=0.005in}
\begin{pspicture}(0,5)(100,40)
\psline[linestyle=solid,linewidth=0.02in](10,10)(10,30)
\psline[linestyle=solid,linewidth=0.02in](10,10)(30,10)
\psline[linestyle=solid,linewidth=0.02in](10,30)(30,30)
\psline[linestyle=solid,linewidth=0.02in](30,10)(30,30)
\psline[linestyle=dashed,linewidth=0.01in](20,10)(20,30)
\psline[linestyle=dashed,linewidth=0.01in](10,20)(30,20)
\psline[linestyle=solid,linewidth=0.03in](10,10)(30,30)
\psline[linestyle=solid,linewidth=0.03in](20,20)(30,10)
\rput[c](20,5){$\E$}
\psline[linestyle=solid,linewidth=0.02in](60,10)(60,30)
\psline[linestyle=solid,linewidth=0.02in](60,10)(80,10)
\psline[linestyle=solid,linewidth=0.02in](60,30)(80,30)
\psline[linestyle=solid,linewidth=0.02in](80,10)(80,30)
\psline[linestyle=dashed,linewidth=0.01in](70,10)(70,30)
\psline[linestyle=dashed,linewidth=0.01in](60,20)(80,20)
\psline[linestyle=solid,linewidth=0.03in](60,10)(80,30)
\psline[linestyle=solid,linewidth=0.03in](60,30)(70,20)
\rput[c](70,5){$\F$}
\end{pspicture}
\end{center}
\caption{The pattern classes $\E$ and $\F$}
\label{E&F}
\end{figure}

\begin{lemma}\label{lem-type1-inflation}A type 1 simple permutation can be inflated as follows:
\begin{itemize}
\item The first point can be inflated by $\E$.
\item The last point can be inflated by $\F$.
\item The points of regions $B$ and $F$ and the points $u$ and $d$ can be inflated by $\Av(21)$.
\item The points of region $D$ can be inflated by $\Av(12)$.
\end{itemize}
\end{lemma}

\begin{proof} Let $\sigma$ be a type 1 simple permutation in $\Av(2143,4231)$ of length $n$. The first and last points can be inflated by $\E$ and $\F$ respectively: for example, since the bottom element of $\sigma$ lies to the right of the leftmost point, inflating the leftmost point by 312 would give rise to a 4231 pattern. Thus the leftmost point must avoid 312, in addition to 2143, but there are no further restrictions.

Every point in the increasing region $F$ can be inflated only by permutations from $\Av(21)$: to see this, note that the last two points of $\sigma$ form a 21-pattern above and to the right of all points in this region, so inflations of points of $F$ must avoid 21 in order not to create a 2143 pattern. A similar argument applies to the points in the increasing region $B$. For the interior points in the decreasing region $D$, note that each forms the `2' of a 321-pattern with the leftmost and rightmost points of $\sigma$. Thus points of $D$ can only be inflated by $\Av(12)$ in order not to create a 4231-pattern.
\end{proof}

The same type of analysis proves:

\begin{lemma}\label{lem-type2-inflation}A type 2 simple permutation of length $n$ can be inflated as follows:
\begin{itemize}
\item The second point (the single point of region $A$) can be inflated by $\E$.
\item The last point can be inflated by $\F$.
\item The first point, the points $d$ and $u$, all points in region $B$ apart from its first, and all points in region $F$ can be inflated by $\Av(21)$.
\item The points of region $D$ can be inflated  by $\Av(12)$.
\item The first point of region $B$ cannot be inflated.
\end{itemize}
\end{lemma}


An analogous argument holds by considering the reverse complement symmetry for type 3 simple permutations. For type 4 permutations, a similar argument yields:

\begin{lemma}\label{lem-type4-inflation}A type 4 simple permutation of length $n$ can be inflated as follows:
\begin{itemize}
\item The second point can be inflated by $\E$.
\item The penultimate point can be inflated by $\F$.
\item The first point, the last point, all points in region $B$ apart from its first, and all points in region $F$ apart from its last can be inflated by $\Av(21)$.
\item The points of region $D$ can be inflated  by $\Av(12)$.
\item The first point of region $B$ and the last point of region $F$ cannot be inflated.
\end{itemize}
\end{lemma}

Now that we know the inflations of the simple permutations covered by Lemma \ref{extremal3142} we can obtain analogous results for the simple permutations of Corollary \ref{extremal2413} by taking inverses.

Finally, we observe that the two sporadic permutations $42513$ and $35142$ behave as type 2 or type 3 permutations, and their inflations are the same.  Indeed, $35142$ is visibly a subpermutation of any of the permutations depicted in Figure \ref{fig-3142-case-result} and so its inflations are subpermutations of their inflations.  A similar property holds (by the reverse complement symmetry) for the inflations of $42513$.

It follows from these proofs and the preceding remarks that the class $\Av(2143,4231)$ is equal to the union of the 2 grid classes represented by the picture:


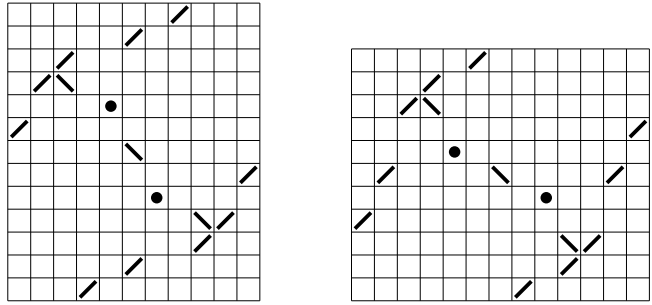
\begin{center}
\psset{xunit=0.006in, yunit=0.006in} \psset{linewidth=0.005in}
\begin{pspicture}(0,0)(560,260)
\multido{\i=0+20}{12}{\psline[linestyle=solid,linewidth=0.005in](\i,0)(\i,260)}
\multido{\i=0+20}{14}{\psline[linestyle=solid,linewidth=0.005in](0,\i)(220,\i)}
\multido{\i=300+20}{14}{\psline[linestyle=solid,linewidth=0.005in](\i,0)(\i,220)}
\multido{\i=0+20}{12}{\psline[linestyle=solid,linewidth=0.005in](300,\i)(560,\i)}
\psline[linestyle=solid,linewidth=0.02in](63,3)(77,17)
\psline[linestyle=solid,linewidth=0.02in](103,23)(117,37)
\psline[linestyle=solid,linewidth=0.02in](163,43)(177,57)
\psline[linestyle=solid,linewidth=0.02in](183,63)(197,77)
\psline[linestyle=solid,linewidth=0.02in](203,103)(217,117)
\psline[linestyle=solid,linewidth=0.02in](3,143)(17,157)
\psline[linestyle=solid,linewidth=0.02in](23,183)(37,197)
\psline[linestyle=solid,linewidth=0.02in](43,203)(57,217)
\psline[linestyle=solid,linewidth=0.02in](103,223)(117,237)
\psline[linestyle=solid,linewidth=0.02in](143,243)(157,257)
\psline[linestyle=solid,linewidth=0.02in](163,77)(177,63)
\psline[linestyle=solid,linewidth=0.02in](103,137)(117,123)
\psline[linestyle=solid,linewidth=0.02in](43,197)(57,183)
\pscircle*(130,90){0.03in}
\pscircle*(90,170){0.03in}
\psline[linestyle=solid,linewidth=0.02in](303,63)(317,77)
\psline[linestyle=solid,linewidth=0.02in](323,103)(337,117)
\psline[linestyle=solid,linewidth=0.02in](343,163)(357,177)
\psline[linestyle=solid,linewidth=0.02in](363,177)(377,163)
\psline[linestyle=solid,linewidth=0.02in](363,183)(377,197)
\psline[linestyle=solid,linewidth=0.02in](403,203)(417,217)
\psline[linestyle=solid,linewidth=0.02in](423,117)(437,103)
\psline[linestyle=solid,linewidth=0.02in](443,3)(457,17)
\psline[linestyle=solid,linewidth=0.02in](483,23)(497,37)
\psline[linestyle=solid,linewidth=0.02in](483,57)(497,43)
\psline[linestyle=solid,linewidth=0.02in](503,43)(517,57)
\psline[linestyle=solid,linewidth=0.02in](523,103)(537,117)
\psline[linestyle=solid,linewidth=0.02in](543,143)(557,157)
\pscircle*(390,130){0.03in}
\pscircle*(470,90){0.03in}

\end{pspicture}
\end{center}

%

%
%
%
%
%
%
%
%
%
%
%
\section{The enumeration of Av(2143,4231)}

We are now able to compute the generating function of the class by using Proposition \ref{CanonicalDecomposition}. This proposition tells us that the class is the disjoint union of the permutation 1, its set of sum decomposable permutations, its skew decomposable permutations, and its inflations of  simple permutations of length at least 4.  Its generating function is therefore the sum of the generating functions of these subsets.

To help in the computation we introduce generating functions for some auxiliary subclasses of $\Av(2143,4231)$.
Let $d(x) = x/(1-x)$ be the generating function for the set of decreasing (or increasing) permutations. Furthermore let $e(x)=\frac{x(1-x)}{1-3x+x^2}$ be the generating function for the set $\E=\Av(312, 2143)$; it is also the generating function for the set $\F=\Av(231, 2143)$ (see e.g. \cite{atkinson:restricted-perm:}).


We will begin by computing $f_\ominus(x)$, the generating function of the skew decomposable permutations of $\Av(2143,4231)$. First note that if we write a skew decomposable permutation $\pi$ as $\pi=\pi_1\ominus\pi_2$, then $\pi_1\in\E$  and $\pi_2\in\F$ (whose generating functions are both $e(x)$). To make this decomposition unique, we insist that $\pi_1$ is skew indecomposable, so we briefly turn our attention to the skew indecomposable permutations in $\Av(312,2143)$, whose generating function we will denote by $e_{\not\ominus}(x)$. The generating function for the  skew decomposable permutations in $\Av(2143,4231)$ will then be given by $f_\ominus(x) = e_{\not\ominus}(x)e(x)$.

Let $\tau$ be an skew indecomposable permutation in $\E = \Av(312,2143)$. As this class contains no simple permutations of length 4 or more, $\tau$ is either the permutation 1 or is sum decomposable. In the latter case, write $\tau=\tau_1\oplus \tau_2$ where $\tau_1$ is  sum indecomposable. If $\tau_1=1$ then $\tau_2$ has no restrictions other than avoiding 312 and 2143, while if $\tau_1$ contains $21$ then $\tau_2$ must be increasing. Recalling that the generating function for the class is $e(x)$, the  skew indecomposable permutations satisfy $e_{\not\ominus}(x)= x + xe(x) + (e(x) - e_{\not\ominus}(x))d(x)$. Solving, yields:

\[ e_{\not\ominus}(x) = \frac{x(1-x)^2}{1-3x+x^2} \]

and so the generating function for the skew decomposables in $\Av(2143,4231)$ is 
\[f_\ominus(x) = \frac{x^2(1-x)^3}{(1-3x+x^2)^2}.\]

We now turn our attention to the  sum decomposable permutations. Consider $\pi=\pi_1\oplus \pi_2$, written so that $\pi_1$ is  sum indecomposable. If $\pi_1=1$, then $\pi_2$ can be any element of the class $\Av(2143,4231)$. Otherwise, $\pi_1$ must contain $21$, so $\pi_2$ must be increasing to avoid 2143.  Thus the generating function $f_{\oplus}(x)$ for the sum decomposable permutations satisfies $f_{\oplus}(x) = xf(x) + (f(x)- f_\oplus(x)-x)d(x)$.

It remains to consider the inflations of the simple permutations of length 4 or more.  As in the previous section we divide the analysis according to the pattern determined by the 4 extremal points and we begin with the case that this pattern is 3142.  We follow our previous analysis of permutation types by defining  the generating functions for the simple permutations of types 1---4 to be $s_1(x), s_2(x), s_3(x)$ and $s_4(x)$ respectively. Note that $s_2(x)=s_3(x)$.

Reading from left to right in the permutation, we encode points of a simple permutation lying in the grid class $\D$ using the three letters $a$, $b$ and $c$: $a$ represents a point in the lowest cell, $b$ a point in the middle and $c$ a point in the top. To enforce uniqueness, we insist that if a point can be encoded by $b$ then it should be. For example, the encoding of $51647283$ is $bacbcacb$. For the permutation to be simple, the word must not contain any factors $aa$, $bb$ or $cc$ (as otherwise these two points will form an interval of size 2), and additionally for type 1 simple permutations the encoding must start with $ba$ and end with $cb$. (Note that for types 2 and 3, we will drop one of these end conditions, while for type 4 we drop both end conditions.)

\paragraph{Type 1 enumeration.} For $n\geq 4$, these permutations are in bijection with words of length $n$ over $\{a,b,c\}$ of the form $ba\dots cb$, and with no factor $aa$, $bb$ or $cc$. There is one word of length 4, namely $bacb$, and one of length 5, $babcb$. Consider a word $w=w_1\dots w_n$ of length $n\geq 6$. If the fourth symbol from the right, $w_{n-3}= a$ or $b$, then the word $w_1\dots w_{n-3}w_{n-1}w_n$ is a valid word of length $n-1$. On the other hand, if $w_{n-3}=c$, then the word $w_1\dots w_{n-4}w_{n-1}w_n$ is a valid word of length $n-2$, where the other omitted symbol $w_{n-2}$ could have been either $a$ or $b$. Hence we obtain the recurrence $s_{1,n} = s_{1,n-1} + 2s_{1,n-2}$, where $s_{1,n}$ denotes the coefficient of $x^n$ in $s_1(x)$. Solving, yields the generating function \[s_1(x) = \frac{x^4}{(1-2x)(1+x)}.\]
Thus the generating function for the number of permutations in $\Av(2143,4231)$ that are inflations of type 1 simple permutations is
\[ f_1(x) = s_1(d(x)) \cdot \frac{e(x)^2}{d(x)^2} = \frac{x^4(1-x)^2}{(1-3x)(1-3x+x^2)^2}.\]

\paragraph{Type 2 enumeration.} Here, we require the four leftmost points of $\sigma$ to take the form of the four leftmost points as depicted in Figure~\ref{fig-3142-case-result}. The remainder of the permutation lies in the grid class $\D$. Thus we enumerate a permutation of length $n$ by considering words of length $n-4$ over $a,b,c$. As before these have no repeated letters as factors, but note that we can now drop the condition that the word starts with $ba$ as the four points placed at the left of the permutation guarantee that no interval can be found here. Thus, for $n\geq 6$ the number $b_n$ of type 2 simple permutations of length $n$ is given by the number of words of length $n-4$ of the form $\dots cb$ with no repeated letters as factors. Reading the word from right to left, this is easily seen to give $2^{n-6}$ choices for $n\geq 6$, giving the generating function $s_2(x)=\frac{x^6}{(1-2x)}$. Consequently, the generating functions for permutations that are inflations of type 2 simple permutations is
\[ f_2(x) = s_2(d(x))\cdot \frac{x\cdot e(x)^2}{d(x)^3} = \frac{x^6}{(1-3x)(1-3x+x^2)^2}.\]

\paragraph{Type 3 enumeration.} Using the reverse complement symmetry, we immediately obtain $s_3(x)=s_2(x)$ and $f_3(x)=f_2(x)$.

\paragraph{Type 4 enumeration.} Here, the first four and last four points of the permutation must be fixed as shown in Figure~\ref{fig-3142-case-result}. We encode the intermediate points as before, but note now that we have no end restrictions at either end. Thus there is one of length 8, while for $n\geq 9$ there are $3\cdot 2^{n-9}$ possible choices, and hence we obtain the generating function $s_4(x)=\frac{x^8(1+x)}{1-2x}$. The generating function for the inflations of these simple permutations is then
\[ f_4(x) = s_4(d(x)) \cdot \frac{x^2 \cdot e(x)^2}{d(x)^4} = \frac{x^8}{(1-3x)(1-x)^2(1-3x+x^2)^2}.\]

Because of the inversion symmetry we obtain exactly the same generating functions when the 4 extremal points have the pattern 2413.

To complete the enumeration we have to consider the permutations $42513$ and $35142$ (when the 4 extremal points have the pattern 3412). They behave like types 2 and 3 respectively and so we find the generating function for inflations of these two sporadic simple permutations to be $2\cdot x \cdot e(x)^2 \cdot d(x)^2= \frac{2x^5}{(1-3x+x^2)^2}$.

This analysis into cases is obviously without overlaps. Thus the generating function $s(x)$ for the simple permutations of length 4 or more in $\Av(2143,4231)$ is $s(x)=2x^5+2(s_1(x)+s_2(x)+s_3(x)+s_4(x))$ and so is given by:
\[s(x) = \frac{2x^4(1+x+x^2+x^4+2x^5+x^6)}{(1-2x)(1+x)}\]

The first few terms of the sequence (starting at $n=4$) are 2, 4, 10, 18, 40, 80, 162.

Likewise, the generating function for the entire class $\Av(2143,4231)$ is:

\[f(x) = x + f_\oplus(x) + f_\ominus(x) + \frac{2x^5}{(1-3x+x^2)^2} + 2(f_1(x)+f_2(x)+f_3(x)+f_4(x))
\]
We have explicit formulae for every term on the right-hand side of this equation except for the term $f_\oplus(x)$.  However, for this term we have an equation that relates it to $f(x)$ itself and solving the resulting equations gives
\[f(x)=\frac{x-11x^2+51x^3-127x^4+186x^5-165x^6+87x^7-23x^8+3x^9}{(1-3x)(1-x)^4(1-3x+x^2)^2}
\]

%
%
%
%
%
%
%
%
%
%
%
\section{Some related pattern classes}

The pattern class we have considered in this paper is the first class of the form $\Av(\alpha,\beta)$ with $|\alpha|=|\beta|=4$ that has been analysed by means of grid classes.  We expect that  grid classes will play an increasingly important role in the study of pattern classes.  A forthcoming paper \cite{albert:forest-and-geometric} gives some very general conditions under which a grid class can be defined by finitely many forbidden permutations, has a rational generating function, and is partially well-ordered.  Furthermore there is already a useful criterion \cite{huczynska:grid-classes-an:} for a given pattern class to be contained within a grid class.  For example, this criterion applies to $\Av(2143, 4321)$,  $\Av(2143,4312)$ and $\Av(1324, 4312)$, all hitherto unenumerated.  We expect to report on these pattern classes in a future paper.

\bibliographystyle{acm}
\bibliography{refs}

\end{document}